\documentclass[11pt,a4paper]{amsart}
\usepackage{graphicx}
\usepackage{amssymb}
\usepackage{epstopdf}
\usepackage{amsmath}
\usepackage{amscd}
\usepackage{amsfonts}

\newtheorem{theorem}{Theorem}
\newtheorem{proposition}[theorem]{Proposition}

\newtheorem*{TheoremA}{Theorem A}

\newcommand\M{\mathcal{M}}

\newcommand\R{\mathcal{R}}

\begin{document}

\title[Boundedness of maximal functions]{Boundedness of maximal functions on non-doubling manifolds with ends}

\author{Xuan Thinh  Duong}

\address{Xuan  Thinh  Duong,
Department of
Mathematics,     Macquarie
University, N.S.W. 2109 Australia}
\email{xuan.duong@mq.edu.au}

\author{Ji Li}
\address{Ji Li, Department of Mathematics, Sun Yat-sen University,
Guangzhou, 510275, China}
\email{liji6@mail.sysu.edu.cn}

\author{Adam Sikora}
\address{Adam Sikora, Department of
Mathematics,     Macquarie
University, N.S.W. 2109 Australia}
\email{adam.sikora@mq.edu.au}

\thanks{This work was started during the second named author's
stay at Macquarie University.  J. Li was  supported by a scholarship from
Macquarie University during 2008-2009, and is supported by China Postdoctoral Science Foundation funded project (Grant No. 201104383) and the
Fundamental Research Funds for the Central Universities (No. 11lgpy56).}
\thanks{{\it {\rm 2010} Mathematics Subject
Classification:}  Primary 42B15;  Secondary 35P99}


\begin{abstract}
Let $M$ be a manifold with ends constructed in \cite{GS} and $\Delta$ be the Laplace-Beltrami operator on $M$.
In this note, we  show the weak type $(1,1)$  and $L^p$ boundedness of the Hardy-Littlewood maximal
function  and of the maximal function associated with the heat semigroup $\M_\Delta f(x)=\sup_{t> 0} |\exp (-t\Delta)f(x)| $
on $L^p(M)$  for $1 < p \le \infty$. The significance of these results comes from the fact that $M$ does not satisfies the doubling condition.
\end{abstract}
\maketitle{}

\section{Introduction}
The theory of Calder\'on-Zygmund operators has played a crucial  role in harmonic analysis and its wide applications
in the last half a century or so. We refer readers to the excellent book \cite {S} and the references therein. In the
standard Calder\'on-Zygmund theory, an essential feature is the so-called doubling condition.
Let us recall that a metric space $(X, d, \mu )$ equipped with a metric $d$ and a measure $\mu$
satisfies the doubling condition if there exists a constant $C$ such that
 $$
 \mu (B(x,2r)) \le C  \mu (B(x,r))
 $$
for all $x\in X$ and $r>0$.

\medskip

Many metric spaces in classical analysis satisfy the doubling condition such as the Euclidean spaces and their smooth domains
(with Lebesgue measure), Lie groups and manifolds of polynomial growth. However, there are significant applications for which underlying ambient spaces do not satisfy the doubling
condition, for example domains of Euclidean spaces with rough boundaries, Lie groups and manifolds with
exponential growth. To these non-doubling spaces, the standard Calder\'on-Zygmund theory
established in the 70's and 80's  is not applicable.

\medskip

Recent works of Nazarov, Treil, Volberg, Tolsa and others, see for example \cite{NTV1, NTV2, NTV3, MMNO, T1, T2} show that
a large part of the standard Calder\"on-Zygmund
theory can be adapted to the case of non-doubling spaces which satisfy a mild growth condition. In
 \cite{DM}, Duong and A. McIntosh also obtain  estimates for certain singular integrals
acting on some domains which do not necessarily  satisfy the doubling condition.
However,  the theory of singular integrals  on
non-doubling spaces  is far from being complete and there are still many significant open problems in this topic.

\medskip

In this note, we study the boundedness of certain maximal functions on non-doubling manifolds with ends.
More specifically, we will show the weak type $(1,1)$  of the Hardy-Littlewood maximal
function and  the maximal function  associated with the heat semigroup of the
Laplace-Beltrami operator as well as  $L^p$ boundedness for these maximal operators
 for $1 < p \le \infty$.
Let us recall that the maximal function associated with the heat semigroup
is defined by the following formula
\begin{equation}\label{mdelta}
\M_\Delta f(x)=\sup_{t> 0} |\exp (-t\Delta)f(x)|
\end{equation}
for $f \in L^p(M)$, $1 \le p \le \infty$.
The behaviour of the kernels of the semigroup $\exp (-t\Delta)$ on manifolds with ends was
studied in \cite{GS}. For the convenience of reader, we recall the main  result of \cite{GS} in the next section as it
plays a key role in our estimates of the operator  $\M_\Delta$.

\section{Manifolds with ends}


Let $M$ be a complete non-compact Riemannian manifold. Let $K\subset M$ be a compact set with non-empty interior and smooth boundary such that $M\backslash K$ has $k$ connected components $E_1,\ldots,E_k$ and each $E_i$ is non-compact. We say in such a case that $M$ has $k$ ends with respect to $K$ and refer to $K$ as the central part of $M$. In many cases, each $E_i$ is isometric to the exterior of a compact set in another manifold $M_i$. In such case, we write $M=M_1\sharp M_2\sharp\cdots\sharp M_k$ and refer to $M$ as a connected sum of the manifolds $M_i$, $i = 1, 2, \cdots, k$.

Following \cite{GS} we consider the following model case. Fix a large integer $N$ (which will be the topological dimension of $M$) and, for any integer $m\in[2,N]$, define the manifold $\R^m$ by
$$
 \R^m=\mathbb{R}^m\times \mathbb{S}^{N-m}.
$$
The manifold $\R^m$ has topological dimension $N$ but its ``dimension at infinity'' is $m$ in the sense that $V(x,r)\approx r^m$  for $r\geq1$, see \cite[(1.3)]{GS}. Thus, for different values of $m$, the manifold $\R^m$ have different dimension at infinity but the same topological dimension $N$, This enables us to consider finite connected sums of the $\R^m$'s.

Fix $N$ and $k$ integers $N_1, N_2,\ldots,N_k\in [2,N]$ such that
$$N=\max\{N_1, N_2,\ldots,N_k\}.$$
Next consider the manifold
$$
    M=\R^{N_1}\sharp\R^{N_2}\sharp\cdots\sharp\R^{N_k}.
$$
In \cite{GS} GrigorÕyan and Saloff-Coste establish both the global upper bound and lower bound for the heat kernel acting on this model class. 
Now we recall the first part of their results with the hypothesis that
$$
    n:=\min_{1\leq i\leq k} N_i >2.
$$
Let $K$ be the central part of $M$ and $E_1,\ldots,E_k$ be the ends of $M$ so that $E_i$ is isometric to the complement of a compact set in $\R^{N_i}$. Write $E_i=\R^{N_i}\backslash K$. Thus, $x\in \R^{N_i}\backslash K$ means that the point $x\in M$ belongs to the end associated with $\R^{N_i}$. For any $x\in M$, define
$$
   |x|:=\sup_{z\in K}d(x,z),
$$
where $d=d(x,y)$ is the geodesic distance in $M$.
One can see that $|x|$ is separated from zero on $M$ and $|x|\approx 1+d(x,K)$.

For $x\in M$, let
$$
   B(x,r):=\{y\in M: d(x,y)<r\}
$$
be the geodesic ball with center $x\in M$ and radius $r>0$ and let
$
   V(x,r)=\mu(B(x,r))
$
where $\mu$ is a Riemannian measure on $M$.

Throughout the paper, we take the simple case $k=2$ for the model of metric spaces with non-doubling measure, i.e., we set $M=\R^n\sharp\R^m$ with $2<n < m$.
Then, from the construction of the manifold $M$, we can see that
\begin{itemize}
\item[(a)] $V(x,r)\thickapprox r^m$ for all $x\in M$, when $r\leq 1$;
\item[(b)] $V(x,r)\thickapprox r^n$ for $B(x,r)\subset \R^n$, when $r> 1$; and
\item[(b)] $V(x,r)\thickapprox r^m$ for $x\in \R^n\backslash K$, $r>2|x|$, or $x\in \R^m$, $r>1$.
\end{itemize}
It is not difficult to check that $M$ does not satisfy the doubling condition.  Indeed, consider a
sequence of balls $B(x_k,r_k)\subset \R^n$ such that $r_k = |x_k| > 1$ and $r_k \rightarrow \infty$ as $k \rightarrow \infty$.
Then $V(x_k,r_k)\thickapprox (r_k)^n$. However, $V(x_k,2r_k)\thickapprox (r_k)^m$ and the doubling condition fails.

Let $\Delta$  be the Laplace-Beltrami operator on $M$ and $e^{-t\Delta}$ the heat semi-group generated by $\Delta$. We denote by $p_t(x,y)$ the heat kernel associated to $e^{-t\Delta}$.

We recall here the following theorem which is  the main results obtain~in~\cite{GS}.
\begin{TheoremA}\cite{GS} Let $M=\R^m\sharp \R^n$ with $2<n < m$. Then the heat kernel $p_t(x,y)$ satisfies the following estimates.

\noindent 1. For $t\leq 1$ and all $x,y\in M$,
\begin{eqnarray*}
p_t(x,y)\approx {C \over V(x,\sqrt{t})}\exp\Big( -c {d(x,y)^2\over t} \Big).
\end{eqnarray*}
2. For $x,y\in K$ and all $t>1$,
\begin{eqnarray*}
p_t(x,y)\approx {C \over t^{n/2}}\exp\Big( -c {d(x,y)^2\over t} \Big).
\end{eqnarray*}
3. For $x\in \R^m\backslash K$, $y\in K$ and all $t>1$,
\begin{eqnarray*}
p_t(x,y)\approx C \Big({1 \over t^{n/2} |x|^{m-2}  }+ {1\over t^{m/2}}\Big)\exp\Big( -c {d(x,y)^2\over t} \Big).
\end{eqnarray*}
4. For $x\in \R^n\backslash K$, $y\in K$ and all $t>1$,
\begin{eqnarray*}
p_t(x,y)\approx C \Big({1 \over t^{n/2} |x|^{n-2}  }+ {1\over t^{n/2}}\Big)\exp\Big( -c {d(x,y)^2\over t} \Big).
\end{eqnarray*}
5. For $x\in \R^m\backslash K$, $y\in \R^n\backslash K$ and all $t>1$,
\begin{eqnarray*}
p_t(x,y)\approx C \Big( {1\over t^{n/2}|x|^{m-2}} + {1\over t^{m/2}|y|^{n-2}} \Big)\exp\Big( -c {d(x,y)^2\over t} \Big)
\end{eqnarray*}
6. For $x,y\in \R^m\backslash K$ and all $t>1$,
\begin{eqnarray*}
p_t(x,y)\approx  {Ct^{-n/2}\over |x|^{m-2}|y|^{m-2}}\exp\Big( -c {|x|^2+|y|^2\over t} \Big) +{C\over t^{m/2}}\exp\Big( -c {d(x,y)^2\over t} \Big)
\end{eqnarray*}
7. For $x,y\in \R^n\backslash K$ and all $t>1$,
\begin{eqnarray*}
p_t(x,y)\approx  {Ct^{-n/2}\over |x|^{n-2}|y|^{n-2}}\exp\Big( -c {|x|^2+|y|^2\over t} \Big) +{C\over t^{n/2}}\exp\Big( -c {d(x,y)^2\over t} \Big).
\end{eqnarray*}

\end{TheoremA}


\section{The boundedness of Hardy-Littlewood maximal function}\label{sec3}
In this section we consider $M=\R^m\sharp\R^n$ for $m > n>2$. A main difficulty which we encounter in
our study is that the doubling condition fails
 in this setting.  However, local doubling still holds, i.e. the doubling condition holds for a ball $B(x,r)$
under the additional assumption $r\le 1$.

\medskip

Let us recall next the standard definition of uncentered Hardy--Littlewood Maximal function.
For any  $p \in [1,\infty]$ and any function $f\in L^p$
let
$$
\M f(x) =\sup_{y\in M,\ r>0} \left\{ \frac{1}{V(y,r)}\int_{B(y,r)} | f(z) | dz\colon x\in B(y,r) \right\}.
$$
Also we have the centered Hardy--Littlewood Maximal function.
For any  $p \in [1,\infty]$ and any function $f\in L^p$
we set
$$
\M_c f(x) =\sup_{r>0} \frac{1}{V(x,r)}\int_{B(x,r)} | f(z) | dz.
$$
It is straightforward to see that $\M_c f(x) \le \M f(x)$ for all $x$. Moreover
in the doubling setting
\begin{equation}\label{cen}
\M(f) \le C\M_c(f),\end{equation} where $C$ is the same constant as in
the doubling condition.
 However, we point out that  estimate \eqref{cen} does not hold in the setting $M=\R^m\sharp\R^n$ with $m > n>2$.
More specifically, one has the following proposition.
\begin{proposition}\label{cenf}
In the setting $M=\R^m\sharp\R^n$ with $m > n>2$, the
estimate $\M(f)  \le C \M_c(f)$ fails for any constant $C$.
\end{proposition}

\begin{proof}
Denote  the characteristic functions of the sets $\R^m\backslash K$, $\R^n\backslash K$ and $K$
by $\chi_1$, $\chi_2$ and $\chi_3$, respectively. Let $f=\chi_2$. Then for any fixed $x\in \R^m$, we
first note that
$${1\over V(B)}\int_B \chi_2(y)dy \le 1$$
for any $B\ni x$. Furthermore, we can construct balls $B\ni x$ such that the ball $B$ with centre $z$, radius $r$, lying mostly in $\R^n$
by choosing $z \in \R^n$, $r$ large enough and $d(z,x) = r-\epsilon$ for $\epsilon$ sufficiently small.
This implies that
$$
  \M(f)(x)=\sup_{B\ni x}{1\over V(B)}\int_B \chi_2(y)dy=1 .
$$
Now consider the centered Hardy--Littlewood Maximal function $\M_c(f)$.
By the definition  for any $r>0$,
$$
    \frac{1}{V(x,r)}\int_{B(x,r)}f(z)dz=    \frac{C}{r^m}\int_{B(x,r)\cap (\R^n\backslash K) }dz.
$$
This implies that $r> |x|$ and the term $\displaystyle {\frac{C}{r^m}\int_{B(x,r)\cap (\R^n\backslash K) }dz }$ is comparable to
$\displaystyle{    {(r-|x|)^n \over r^m}}.$

It is easy to check  that the maximal value of the above term is comparable to
$$
     \big({n|x|\over m-n}\big)^n  \big\slash \big({m|x|\over m-n}\big)^m,
$$
which shows that $\M(f)$ is not pointwise bounded by any multiple of $\M_c(f)$ since the maximal value depends on $x$ and tends to zero when $|x|$ goes~to~$\infty$. This proves Proposition~\ref{cenf}.
\end{proof}

\begin{theorem}\label{max}
The maximal function $\M(f)$ is of weak type $(1,1)$ and bounded on all $L^p$ spaces for $1< p \le \infty$.
\end{theorem}
\begin{proof} Here and throughout the paper, for the sake of simplicity we use $|\cdot|$ to denote the measure of the sets in $M$.
It is straightforward that the maximal function $\M(f)$ is bounded on $L^\infty$. We will show that the weak type $(1,1)$
estimate
$$
|\{x \colon \M f(x) > \alpha \}| \le C \frac{\|f\|_1}{\alpha}
$$
holds, then the $L^p$ boundedness of $\M(f)$ follows from the Marcinkiewicz interpolation theorem.


We consider two cases:

\medskip
\noindent
{\bf Case 1:}
$\frac{\|f\|_1}{\alpha} < 1$.
\medskip

\noindent
Following the standard proof of weak type for Maximal operator we note that for any
$x\in \{x \colon \M f(x) > \alpha \}$ there exist a ball such that $x\in B(y,r)$ and
\begin{eqnarray}\label{w11}
 \frac{1}{V(y,r)}\int_{B(y,r)} |f(z)| dz > \alpha.
\end{eqnarray}
This implies $$\| f \| _1 = \int _M |f(z)| dz \ge \int_{B(y,r)} |f(z)| dz > \alpha V(y,r) .$$
Therefore  $1>\frac{\|f\|_1}{\alpha}>{V(y,r)}$, hence $r \le 1$ and the ball $B(y,r)$ satisfies doubling condition so
one can use standard Vitali covering argument to prove weak type $(1,1)$ estimate in this case.

\medskip
\noindent
{\bf Case 2 :} $\frac{\|f\|_1}{\alpha} \ge 1$.

\medskip
\noindent
First we split  $M$ into three components $\R^m\backslash K$, $\R^n\backslash K$ and $K$, and denote their characteristic functions by
$\chi_1$, $\chi_2$ and $\chi_3$, respectively.
Since the maximal function $\M(f)$ is sublinear, it is enough to show
 that
each of the three terms $\M( \chi_1 f)$, $\M( \chi_2 f)$ and $\M( \chi_3 f)$ is of weak type $(1,1)$.

\medskip

We first consider $\M( \chi_1 f)$. Then
\begin{align*}
&|\{x: \M (\chi_1f)(x) > \alpha \}|
 \leq |\{x \in \R^m\backslash K: \M (\chi_1f)(x) > \alpha \}| \\
&\hskip.5cm +  |\{x \in \R^n\backslash K: \M (\chi_1f)(x) > \alpha \}|+ |\{x \in K: \M (\chi_1f)(x) > \alpha \}| \\
&\hskip.5cm =: I_1 +I_2 +I_3.
\end{align*}

The estimate for  $I_1$ follows from the classical weak type $(1,1)$ estimate since $\chi_1f$ is a function on $\R^m\backslash K$ and the measure on $\R^m\backslash K$ satisfies
the doubling condition.

To estimate  $I_2$,  we note that for all $ x\in \R^n\backslash K$,
$$\sup  \left\{\frac{1}{|B(y,r)|}\colon r>d(x,y)\quad \mbox{and}\quad B(y,r)\cap (\R^m\backslash K) \neq \emptyset  \right\} \le C\frac{1}{|x|^n}.
$$
The above inequality implies that
\begin{equation}\label{aa}
\M \chi_1f(x) \le C\frac{\|\chi_1f\|_1}{|x|^n} \quad \forall x\in \R^n\backslash K.
\end{equation}
Hence
$$I_2 \leq  |\{x \in \R^n\backslash K: C\frac{\|\chi_1f\|_1}{|x|^n} > \alpha \}| \leq C \frac{\|\chi_1f\|_1}{\alpha}\leq  C \frac{\|f\|_1}{\alpha}.$$

To estimate $I_3$, we note that the measure of $K$ is finite. Therefore
$$
|I_3|\leq |K|\leq C \frac{\|f\|_1}{\alpha}.
$$

To prove the weak $(1,1)$ estimate of $\M( \chi_2 f)$ we note that
\begin{align*}
&|\{x: \M (\chi_2f)(x) > \alpha \}|\leq |\{x \in \R^m\backslash K: \M (\chi_2f)(x) > \alpha \}|\\
&\hskip.5cm  +  |\{x \in \R^n\backslash K: \M (\chi_2f)(x) > \alpha \}| + |\{x \in K: \M (\chi_2f)(x) > \alpha \}| \\
&\hskip.5cm =: I\!I_1 +I\!I_2 +I\!I_3.
\end{align*}
 $I\!I_2$ and $I\!I_3$ can be verified following the same steps as for  $I_1$ and $I_3$, respectively. To estimates $I\!I_1$ we observe that
\begin{equation}\label{bb}
\quad \M \chi_2f(x) \le C\frac{\|\chi_2f\|_1}{|x|^m} \quad \forall x\in \R^m\backslash K.
\end{equation}
Hence $I\!I_2\leq C \frac{\|f\|_1}{\alpha}.$

Similarly, to deal with  $\M( \chi_3 f)$ we note that
\begin{align*}
&|\{x: \M (\chi_3f)(x) > \alpha \}|\leq |\{x \in \R^m\backslash K: \M (\chi_3f)(x) > \alpha \}| \\
&\hskip.5cm +  |\{x \in \R^n\backslash K: \M (\chi_3f)(x) > \alpha \}| + |\{x \in K: \M (\chi_3f)(x) > \alpha \}| \\
&\hskip.5cm =: I\!I\!I_1 +I\!I\!I_2 +I\!I\!I_3.
\end{align*}
The estimate of $I\!I\!I_1$ follows immediately since the measure on $(\R^m\backslash K) \cup K$ satisfies the  doubling condition.
The estimate of $I\!I\!I_3$ is the same as that of $I_3$ or $I\!I_3$. Next to estimates $I\!I\!I_2$ we further decompose $\{x \in \R^n\backslash K\}$ into two parts $\{x \in \R^n\backslash K: |x|\leq 2\}$ and $\{x \in \R^n\backslash K: x>2\}$. For the first part we directly have
$$
|\{x \in \R^n\backslash K: |x|\leq 2,\ \M (\chi_3f)(x) > \alpha \}|\leq C \leq   C \frac{\|f\|_1}{\alpha}.
$$
For the second part, similar to the estimate of $I_2$, we note that for all $x\in \R^n\backslash K$ and $|x|>2$,
$$
\sup  \left\{\frac{1}{|B(y,r)|}\colon r>d(x,y) \quad \mbox{and}\quad B(y,r)\cap K \neq \emptyset  \right\} \le C\frac{1}{|x|^n}.
$$
Hence,
\begin{equation*}
\M \chi_3f(x) \le C\frac{\|\chi_3f\|_1}{|x|^n} \quad \forall x\in \R^n\backslash K\ \textup{and}\ |x|>2,
\end{equation*}
which implies that
$$
|\{x \in \R^n\backslash K: |x|> 2,\ \M (\chi_3f)(x) > \alpha \}|\leq   C \frac{\|f\|_1}{\alpha}.
$$

Combining the estimates of $\M( \chi_1 f)$, $\M( \chi_2 f)$ and $\M( \chi_3 f)$ we verify   (\ref{w11}).
The proof of Theorem \ref{max} is now complete.


\end{proof}

\section{The boundedness of the maximal function $\M_\Delta$}

In this section  we prove that the heat maximal operator satisfies weak type $(1,1)$  and  is bounded on $L^p$ for $1 < p \le \infty$.

\medskip

We note that when the heat semigroup has a Gaussian upper bound, then the maximal function corresponding
to heat semigroup  is pointwise dominated by the Hardy-Littlewood maximal operator. In this case, the
weak type $(1,1)$ estimate of $\M_\Delta $ follows from the weak type $(1,1)$ estimate of the Hardy-Littlewood
maximal function.
However, in considered  setting this is no longer the case and the operator
$\M_\Delta $  can not be controlled
by the Hardy-Littlewood maximal function. We can see this via the estimates of the heat semigroup  in the proof of
Theorem \ref{theorem2} below where we give a direct proof of the weak type estimates of the heat maximal operator.

\medskip

The following theorem is the main result of this section.
\begin{theorem}\label{theorem2}
Let $\M_\Delta$ be the operator defined by \eqref{mdelta}. Then $\M_\Delta$ is weak type $(1,1)$ and for any
function $f\in L^p$, $1<p\leq\infty$, the following estimates hold
$$
\|\M_\Delta f\|_{L^p(M)} \le C  \|f\|_{L^p(M)}.
$$
\end{theorem}
\begin{proof}

We first show that $\M_\Delta$ is weak type $(1,1)$, i.e., we need to prove
that there exists a positive constant $C$ such that for any $f\in
L^1(M)$ and for any $\lambda>0$,
\begin{eqnarray}\label{weak 1-1}
\big|\big\{ x\in M: \sup\limits_{t>0} |\exp(-t\Delta)f(x)|>\lambda
\big\}\big|\leq {C\over \lambda}\|f\|_{L^1(M)}.
\end{eqnarray}

Fix $f\in L^1(M)$. Similarly as in Section \ref{sec3} we set $f_1(x)=f(x)\chi_{{\R}^m\backslash
 {K}}(x)$, $f_2(x)=f(x)\chi_{{\R}^n\backslash
{K}}(x)$ and $f_3(x)=f(x)\chi_{ K }(x)$, where $ K $
is the center of $M$. To prove (\ref{weak 1-1}), it suffices
to verify that the following three estimates hold:
\begin{eqnarray}\label{w 1}
\big|\big\{ x\in {\R}^m\backslash K :
\sup\limits_{t>0} |\exp(-t\Delta)f(x)|>\lambda \big\}\big|\leq
{C\over \lambda}\|f\|_{L^1(M)};
\end{eqnarray}
\begin{eqnarray}\label{w 2}
\big|\big\{ x\in {\R}^n\backslash K :
\sup\limits_{t>0} |\exp(-t\Delta)f(x)|>\lambda \big\}\big|\leq
{C\over \lambda}\|f\|_{L^1(M)};
\end{eqnarray}
\begin{eqnarray}\label{w 3}
\big|\big\{ x\in  K : \sup\limits_{t>0}
|\exp(-t\Delta)f(x)|>\lambda \big\}\big|\leq {C\over
\lambda}\|f\|_{L^1(M)}.
\end{eqnarray}

We first consider (\ref{w 1}). Since $\M_\Delta$ is a sublinear operator, we
have
\begin{eqnarray*}
&&\big|\big\{ x\in {\R}^m\backslash K :
\sup\limits_{t>0} |\exp(-t\Delta)f(x)|>\lambda \big\}\big| \\
&&\leq \big|\big\{ x\in {\R}^m\backslash K :
\sup\limits_{t>0} |\exp(-t\Delta)f_1(x)|>\lambda
\big\}\big| \\
&&\hskip.5cm +\big|\big\{ x\in {\R}^m\backslash K :
\sup\limits_{t>0} |\exp(-t\Delta)f_2(x)|>\lambda
\big\}\big|\\
&&\hskip.5cm+\big|\big\{ x\in {\R}^m\backslash K :
\sup\limits_{t>0} |\exp(-t\Delta)f_3(x)|>\lambda \big\}\big|\\
&&=: I_1+I_2+I_3.
\end{eqnarray*}
To estimate  $I_1$ we consider two cases.

{\bf Case 1}: $t>1$. By Theorem A Point 6
 \begin{align*}
|\exp(-t\Delta)f_1(x)| &\leq C\int_{{\R}^m\backslash K }
\Big( {1\over
t^{n\over2}|x|^{m-2}|y|^{m-2}}\exp(-{c(|x|^2+|y|^2)\over t}) \\
&\qquad\qquad\qquad +
{1\over t^{m\over2}}\exp(-{cd(x,y)^2\over t})\Big) |f(y)|dy\\
&=:  I_{11}+I_{12}.
\end{align*}

To estimate  $I_{11}$ we note that
\begin{eqnarray*}
{t^{-n/2}\over |x|^{m-2}|y|^{m-2}}\exp(-{c(|x|^2+|y|^2)\over
t})\leq C{t^{-n/2}\over |x|^{m-2}|y|^{m-2}}{t^{n\over2}\over
(t+|x|^2+|y|^2)^{n\over2}}\\ \le {1\over |x|^{m-2+n}}\leq {1\over |x|^{m}}
\end{eqnarray*}
since $|y|\geq1$ and $n>2$. Hence,
\begin{eqnarray*}
I_{11}\leq C\int_{{\R}^m\backslash K } {1\over |x|^{m-2+n}} f(y)dy \leq C{\|f\|_{L^1(M)}\over |x|^m}. \\
\end{eqnarray*}

To estimate  $I_{12}$ we note that if $x\in {\R}^m\backslash K  $
then
\begin{eqnarray*}
\int_{{\R}^m\backslash K }{1\over t^{m\over2}}\exp(-{cd(x,y)^2\over t}) |f(y)|dy
\le C\mathcal{M}_{{\R}^m\backslash K }(f)(x)
\end{eqnarray*}
where $\mathcal{M}_{{\R}^m\backslash K }(f)(x)$ is the Hardy-Littlewood maximal
function acting on ${\R}^m\backslash K $.
\medskip

{\bf Case 2}: $t\leq1$. By Theorem A Point 1
\begin{eqnarray*}
 |\exp(-t\Delta)f_1(x)|\leq \int_{{\R}^m\backslash K }
{1\over t^{m\over2}}\exp(-{cd(x,y)^2\over t})   |f(y)|dy.
\end{eqnarray*}
Again   the right-hand side of above estimate is bounded by $\mathcal{M}_{{\R}^m\backslash K }(f)(x)$.
These estimates prove weak type $(1,1)$ for $I_1$ since ${\R}^m\backslash K $
satisfies doubling condition.

Next we show weak type estimates for $I_2$. We also consider two cases.
\medskip

{\bf Case 1}: $t>1$. By Theorem A Point 5
\begin{align*}
 |\exp(-t\Delta)f_2(x)|&\leq C\int_{{\R}^n\backslash K }
\Big( {1\over t^{n\over2}|x|^{m-2}} +
{1\over t^{m\over2} |y|^{n-2}}\Big)\\
&\qquad\qquad\qquad\qquad\exp(-{cd(x,y)^2\over t}) |f(y)|dy\\
&=: I_{21}+I_{22}.
\end{align*}

Similarly as in the  estimate for $I_{11}$ we get
\begin{eqnarray*}
I_{21}&\leq&  C\int_{{\R}^n\backslash K } {1\over
t^{n\over2}|x|^{m-2}} {t^{n\over2}\over (t+d(x,y)^2)^{n\over2} }
|f(y)|dy\\
&\leq& C\int_{{\R}^n\backslash K } {1\over
|x|^{m-2+n}}
 |f(y)|dy\leq C {\|f\|_{1}\over |x|^m},
\end{eqnarray*}
since $n>2$, $|x|\geq1$ and in this case, $d(x,y)\geq |x|$.

To estimate  $I_{22}$ we note that
\begin{eqnarray*}
I_{22}&\leq&  C\int_{{\R}^n\backslash K } {1\over
t^{m\over2}|y|^{n-2}} {t^{m}\over (t+d(x,y)^2)^{m} }
|f(y)|dy\\
&\leq& C\int_{{\R}^n\backslash K } {t^{m\over2}\over
(t+d(x,y)^2)^{m} }
 |f(y)|dy\\
&\leq& C\int_{{\R}^n\backslash K } {\sqrt{t}^{m}\over
(\sqrt{t}+d(x,y))^{2m} }
 |f(y)|dy
\end{eqnarray*}
since $|y|\geq1$. By decomposing the Poisson kernel
${\displaystyle {\sqrt{t}^{m}\over (\sqrt{t}+d(x,y))^{2m} } }$ into annuli, it is
easy to see that the last term of the above inequality is bounded by
$C\mathcal{M}_{{\R}^n\backslash K}(f)(x)$.

\medskip

{\bf Case 2}: $t\leq1$.
Again by Theorem A Point 1
\begin{eqnarray*}
 |\exp(-t\Delta)f_2(x)|\leq C\int_{{\R}^n\backslash K }
{1\over t^{m\over2}}\exp(-{cd(x,y)^2\over t})   |f(y)|dy.
\end{eqnarray*}
Hence it is bounded by $C\mathcal{M}(f)(x)$.

Similar to $I_1$, we have
\begin{eqnarray*}
I_2\leq C {\|f\|_{1}\over \lambda}.
\end{eqnarray*}

Now we consider  $I_3$.

{\bf Case 1}: $t>1$.
By Theorem A Point 3
\begin{eqnarray*}
 |\exp(-t\Delta)f_3(x)|&\leq& C\int_{ K }
\Big( {1\over t^{n\over2}|x|^{m-2}} +
{1\over t^{m\over2} }\Big)\exp(-{cd(x,y)^2\over t}) |f(y)|dy\\
&=:& I_{31}+I_{32}.
\end{eqnarray*}
To estimate  $I_{31}$ we note that
\begin{eqnarray*}
I_{31}&\leq&  C\int_{ K }
{1\over t^{n\over2}|x|^{m-2}}
{t^{n\over2}\over (t+d(x,y)^2)^{n\over2}} |f(y)|dy\leq C{\|f\|_{1}\over |x|^{m+n-2}}\\
&\leq & C{\|f\|_{1}\over |x|^m},
\end{eqnarray*}
where we use the facts that $n>2$, $|x|>1$ and that in this case, $d(x,y)\approx |x|$. Similarly,
\begin{eqnarray*}
I_{32}&\leq&  C\int_{ K }
{1\over t^{m\over2}}
{t^{m\over2}\over (t+d(x,y)^2)^{m\over2}} |f(y)|dy\leq C{\|f\|_{1}\over |x|^{m}}.
\end{eqnarray*}

{\bf Case 2}: $t\leq1$.
By Theorem A Point 1
\begin{eqnarray*}
 |\exp(-t\Delta)f_3(x)|\leq \int_{ K }
{1\over t^{m\over2}}\exp(-{cd(x,y)^2\over t})   |f(y)|dy.
\end{eqnarray*}
Hence it is bounded by $C\mathcal{M}(f)(x)$.

Combining the estimates of the two cases, we obtain
\begin{eqnarray*}
I_3\leq C {\|f\|_{L^1(M)}\over \lambda}.
\end{eqnarray*}

The estimates of $I_1$, $I_2$ and $I_3$ together imply (\ref{w 1}).

\medskip

We now turn to the estimate of (\ref{w 2}). Similarly to the proof of $(\ref{w 1})$, we have
\begin{eqnarray*}
&&\big|\big\{ x\in {\R}^n\backslash K :
\sup\limits_{t>0} |\exp(-t\Delta)f(x)|>\lambda \big\}\big|\\
&&\leq \big|\big\{ x\in {\R}^n\backslash K :
\sup\limits_{t>0} |\exp(-t\Delta)f_1(x)|>\lambda
\big\}\big|\\
&&\hskip.5cm+\big|\big\{ x\in {\R}^n\backslash K :
\sup\limits_{t>0} |\exp(-t\Delta)f_2(x)|>\lambda
\big\}\big|\\
&&\hskip.5cm+\big|\big\{ x\in {\R}^n\backslash K :
\sup\limits_{t>0} |\exp(-t\Delta)f_3(x)|>\lambda \big\}\big|\\
&&=: II_1+II_2+II_3.
\end{eqnarray*}

We note that the estimate of $II_1$ is similar to that of $I_2$, while the estimate of $II_2$ is similar to that of $I_1$.
Moreover, the estimate of $II_3$ is similar to that of $I_3$. Therefore we can verify that (\ref{w 2}) holds.

Finally, we turn to the estimate of (\ref{w 3}). We have
\begin{eqnarray*}
&&\big|\big\{ x\in  K :
\sup\limits_{t>0} |\exp(-t\Delta)f(x)|>\lambda \big\}\big|\\
&&\leq \big|\big\{ x\in  K :
\sup\limits_{t>0} |\exp(-t\Delta)f_1(x)|>\lambda
\big\}\big|\\
&&\hskip.5cm+\big|\big\{ x\in  K :
\sup\limits_{t>0} |\exp(-t\Delta)f_2(x)|>\lambda
\big\}\big|\\
&&\hskip.5cm+\big|\big\{ x\in  K :
\sup\limits_{t>0} |\exp(-t\Delta)f_3(x)|>\lambda \big\}\big|\\
&&=: III_1+III_2+III_3.
\end{eqnarray*}

Also, we point out that the estimate of $III_1$ is similar to that of $I_3$ and that the estimate of $III_2$ is similar to that of $II_3$.

Concerning the term $III_3$, we first note that in this case $x\in  K $.
We have
\begin{eqnarray*}
 |\exp(-t\Delta)f_3(x)|\leq C\int_{ K } {1\over t^{m\over2}}\exp({-{cd(x,y)^2\over t}})|f(y)|dy.
\end{eqnarray*}
It is easy to see that the right-hand side of the above inequality is bounded by $C\mathcal{M}(f)(x)$. Thus, we have
\begin{eqnarray*}
III_3\leq C {\|f\|_{1}\over \lambda}.
\end{eqnarray*}

Hence, we can see that (\ref{w 3}) holds. Now (\ref{w 1}), (\ref{w 2}) and (\ref{w 3}) together imply that (\ref{weak 1-1}) holds, i.e., $\M_\Delta$ is of weak type (1,1).

Next, note that the semigroup $\exp(-t\Delta)$ is submarkovian   so $\M_\Delta$ is bounded on $L^\infty(M)$. This together with (\ref{weak 1-1}), implies that $\M_\Delta$ is bounded on $L^p(M)$ for all $1<p<\infty$.

The proof of Theorem \ref{theorem2} is complete.


\end{proof}


\bibliographystyle{amsplain}

\end{document}